\documentclass[12pt,a4paper]{amsart}
\usepackage{amsfonts,amssymb,amscd,amsmath,enumerate,verbatim,calc}
\usepackage{mathrsfs}
\numberwithin{equation}{section}
\newtheorem{thm}{Theorem}[section]
\newtheorem{cor}[thm]{Corollary}
\newtheorem{lem}[thm]{Lemma}

\newtheorem{rem}[thm]{Remark}

\DeclareMathOperator{\Ann}{Ann} \DeclareMathOperator{\Tor}{Tor}
\DeclareMathOperator{\Ext}{Ext} \DeclareMathOperator{\Supp}{Supp}
\DeclareMathOperator{\V}{V} \DeclareMathOperator{\Hom}{Hom}
\DeclareMathOperator{\Ker}{Ker} \DeclareMathOperator{\Coker}{Coker}
\DeclareMathOperator{\Image}{Im} 
\DeclareMathOperator{\cd}{cd} 

 \DeclareMathOperator{\Max}{Max}
\DeclareMathOperator{\lc}{H} 
 
\DeclareMathOperator{\G}{\Gamma} 
 
\DeclareMathOperator{\h}{H}

\DeclareMathOperator{\K}{K}
\DeclareMathOperator{\ara}{ara}

\newcommand{\fa}{\mathfrak{a}}

\newcommand{\fp}{\mathfrak{p}}

\newcommand{\mS}{\mathcal S}
\newcommand{\mC}{\mathcal C}

\newcommand{\lo}{\longrightarrow}

\begin{document}

\title[A characterization result of cofinite local cohomology modules]
{A characterization result of cofinite local cohomology modules}

\bibliographystyle{amsplain}

   \author[M. Aghapournahr]{Moharram Aghapournahr}
\address{Department of Mathematics, Faculty of Science, Arak University,
Arak, 38156-8-8349, Iran.}
\email{m-aghapour@araku.ac.ir}

 \author[L. Melkersson]{Leif Melkersson}
\address{Department of Mathematics, University of Lund\\ 
	S-221 00 Lund\\ 
	Sweden}
\email{leif.melkersson@math.lu.se}





\begin{abstract}
Let $R$ be a commutative Noetherian ring, $\fa$ an ideal of $R$,
$M$ an arbitrary $R$-module and $N$ a finite $R$-module. We prove that \cite[Theorem 2.1]{Mel} and \cite[Proposition 3.3 (i)$\Leftrightarrow$(ii)]{B1} are true for any Serre subcategory of $R$-modules. We also prove a characterization theorem for $\lc_{\fa}^{i}(M)$ and $\lc_{\fa}^{i}(N,M)$ to be $\fa$-cofinite for all $i$, whenever one of the following cases holds:
(a) $\ara (\fa)\leq 1$, (b) $\dim R/\fa \leq 1$ or (c) $\dim R\leq 2$. In the end we study Artinianness and Artinian $\fa$-cofiniteness of local cohomology modules.
\end{abstract}

\maketitle


\section{Introduction}


Throughout  this paper $R$ is a commutative Noetherian ring with non-zero
identity and $\fa$ an ideal of $R$. For an $R$-module $M$,  the $i^{th}$
local cohomology module $M$ with respect to ideal $\fa$ is defined as
\begin{center}
	$\lc^{i}_{\fa}(M) \cong \underset{n}\varinjlim \Ext^{i}_{R}(R/{\fa}^{n},M).$
\end{center}

Hartshorne in \cite{Har} defined a module $M$ to be $\fa$--cofinite if $\Supp_R(M) \subseteq \V(\fa)$ and $\Ext^i_
R(R/\fa,M)$ is finite for all $i\geq 0$. He also asked the following question:\\

\noindent {\bf Hartshorne's question:} {\it Let $M$ be a finite $R$-module and $\fa$ be an ideal of $R$. When are
	$\lc^{i}_{\fa}(M)$ $\fa$--cofinite for
	all $i\geq 0$?}\\
The answer is negative in general, see \cite{Har} for a counterexample,  but it is true in the following cases:

\begin{thm}\label{1.1}
	Let $M$ be a finite $R$-module and suppose one of the following cases holds:
	\begin{itemize}
		\item[(a)] $\cd(\fa,R)\leq 1$ or $\dim R\leq 2$. See \cite{Mel}.
		\item[(b)] $\dim R/{\fa}\leq 1$. See \cite{BN}.
	\end{itemize}
	Then $\lc^{i}_{\fa}(M)$ is $\fa$--cofinite for
	all $i\geq 0$.
\end{thm}
Note that $\cd(\fa,R)\leq 1$  (the cohomological dimension of $\fa$ in $R$) is the smallest integer $n$ such that the local
cohomology modules $\lc^{i}_{\fa}(S)$ are zero for all $R$--modules $S$ and for all $i> n$.

The generalized local
cohomology module
\begin{center}
	$\lc^{i}_{\fa}(M,N) \cong \underset{n}\varinjlim
	\Ext^{i}_{R}(N/{\fa}^{n}N,M).$
\end{center}
for all $R$--module $M$ and $N$ was introduced by Herzog in \cite{He}.
Clearly it is a generalization of ordinary local cohomology module. In
\cite{Ya} Yassemi asked whether the  Hartshorne's question hold for generalized local cohomology.
It is answered in the following cases:

\begin{thm}\label{1.2}
	Let $M, N$ be two finite $R$-modules and suppose one of the following cases holds:
	
	\begin{itemize}
		\item[(a)] $\cd(\fa,R)\leq 1$ or $\dim R\leq 2$. See \cite{MS1}.
		\item[(b)] $\dim R/{\fa}\leq 1$. See \cite{DH}.
	\end{itemize}
	Then $\lc^{i}_{\fa}(N,M)$ is $\fa$--cofinite for
	all $i\geq 0$.
\end{thm}

Recall that an arbitrary $R$-module $X$ is said to be  weakly Laskerian if  the set of associated prime ideals of any quotient module of $X$ is finite (see \cite{DiM}). Recently, in \cite{hung}, the author introduced the class of extension modules of finitely generated
modules by the class of all modules  of finite support and named it ${\rm FSF}$
modules. Bahmanpour in  \cite[Theorem 3.3]{B} proved that over a Noetherian ring $R$ an $R$-module
$X$ is weakly Laskerian if and
only if it is ${\rm FSF}$. 


 Recall too that an arbitrary $R$-module $X$ is said to be an $\operatorname{FD}_{<  n}$ $R$-module if there exists a finite submodule  $X'$ of $X$ such that $\dim_R(X/X')< n$ (see \cite{AB}). Recall also that a class of $R$-modules is said to be a {\it Serre subcategory of the category of $R$-modules}, when it is closed under taking submodules, quotients and extensions.\\

 The second author of present paper in \cite[Theorem 2.1]{Mel}
proved that for any ideal $\fa$  of $R$ and any $R$-module $M$, the $R$-module
$\Ext^{i}_{R}(R/\fa,M)$ is finite for all $i$ if and only if the $R$-module $\Tor^R_{i}(R/\fa,M)$ is finite for all $i$ if and only if  Koszul cohomology $\h^i(x_1\dots,x_s ; M)$ is finite for all $i$ where $\mathfrak a=(x_1,\dots,x_s)$. Our main result in Section 2 is in this direction. More precisely, we prove the following result that show \cite[Theorem 2.1]{Mel} and \cite[Proposition 3.3 (i)$\Leftrightarrow$(ii)]{B1} are true for any Serre subcategory of modules:

\begin{thm}\label{T:E}
	Let $\mathcal{S}$ be a Serresubcategory of the category of modules over the 
	Noetherian ring $R$ and let $\mathfrak a=(x_1,\dots,x_s)$ be an ideal of $R$.
	Given an $R$-module $M$ the following conditions are equivalent:
	\begin{itemize}
		\item[(i)] $\Ext_R^{i}(R/{\mathfrak a}, M)$ belongs to $\mathcal S$ 
		for all $i$.
		\item[(ii)] $\Tor_i^R(R/{\fa}, M)$ belongs to $\mS$ for all $i$.
		\item[(iii)] $\h^i(x_1\dots,x_s ; M)$ belongs to $\mS$ for all $i$.
		\item [(iv)]  $\h_i(x_1,\dots,x_s; M)$ belongs to $\mS$ for all $i$.
	\end{itemize} 
\end{thm}

In section 3, we study relationship between cofiniteness of local cohomology and generalized local cohomology modules and we prove that the cofiniteness of these two  kinds of local cohomolgy under the following assumptions depends to each other. More generally we prove a characterization result (see Corollary \ref{min}), in the cases  (a) $\ara (\fa)\leq 1$, (b) $\dim R/\fa \leq 1$ or (c) $\dim R\leq 2$ for $\lc_{\fa}^{i}(M)$ and $\lc_{\fa}^{i}(N,M)$ to be $\fa$-cofinite for all $i$.\\

One of the important question in the study of local cohomology modules is determine when the local cohomology is Artinian (see \cite{Hu})? In this direction, in Section 4,  we study the Artinianness and Artinian $\fa$-cofiniteness of local cohomology modules from upper bound when $\dim R/{\fa} =1$.  More precisely, we show that:


\begin{thm}
	Let $\fa$ be an ideals of a Noetherian ring $R$, with $\dim R/{\fa}=1$. For an $R$-module $M$   and  integer  an  $r$ the following conditions are equivalent; 
	
	\begin{itemize}
		\item[(i)]$\Supp_R(\lc^{i}_\fa(M))\subseteq \Max(R)$ for each $i>r$.
		\item[(ii)] $\lc^{i}_\fa(M)$ is an Artinian $R$-module for each $i>r$.
		\item[(iii)] $\lc^{i}_\fa(M)$ is an Artinian $\fa$-cofinte $R$-module for each $i>r$.
	\end{itemize} 
\end{thm}

Throughout this paper, $R$ will always be a commutative Noetherian ring with
non-zero identity and $\fa$ will be an ideal of $R$.  We denote $\{\frak p \in {\rm
Spec}\,R:\, \frak p\supseteq \fa \}$ by $\V(\fa)$. 
For any unexplained notation and terminology we refer the reader to \cite{BSh}, \cite{BH} and \cite{Mat}.

\maketitle
\section{Equivalence of $\Ext$ and $\Tor$ in Serre subcategories}  

We begin this section with a useful remark that we need in the proof of main result of this section. 

\begin{rem}\label{r:colpow}
	Observe that if $\mS$ is a Serre subcategory of the category of $R$-modules
	then for any $R$-module $M$, such that the submodule  $0\underset{M}{:}\fa$ 
	belongs to $\mS$, then for each $l$ also 
	$0\underset{M}{:}\fa^l$  belongs to $\mS$ and if the quotient module
	$M/{{\fa}M}$ belongs to $\mS$, then $M/{{\fa}^lM}$ belong to $\mS$ 
	for all $l$.
	This follows by induction on $l$.
	Let $x_1,\dots,x_s$ be generators of $\fa$. 
	For the first assertion we use the map 
	${0\underset{M}{:}\fa^{l+1}}
	\to{(0\underset{M}{:}\fa^l)_1^s}$, defined by 
	$u\mapsto(x_iu)_1^s$. This map has the kernel $0\underset{M}{:}{\fa}$.
	For the second assertion use the map 
	$(M/{\fa^{l}M})_1^s\to M/{\fa^{l+1}M}$ defined by 
	$({\overline u}_i)_1^s\mapsto{\sum_1^s \overline{x_iu_i}}$, whose cokernel is
	isomorphic to $M/{\fa M}$.

\end{rem}

\begin{thm}\label{T:E}
	Let $\mathcal{S}$ be a Serresubcategory of the category of modules over the 
	Noetherian ring $R$ and let $\mathfrak a=(x_1,\dots,x_s)$ be an ideal of $R$.
	Given an $R$-module $M$ the following conditions are equivalent:
	\begin{itemize}
		\item[(i)] $\Ext_R^{i}(R/{\mathfrak a}, M)$ belongs to $\mathcal S$ 
		for all $i$.
		\item[(ii)] $\Tor_i^R(R/{\fa}, M)$ belongs to $\mS$ for all $i$.
		\item[(iii)] $\h^i(x_1\dots,x_s ; M)$ belongs to $\mS$ for all $i$.
		\item [(iv)]  $\h_i(x_1,\dots,x_s; M)$ belongs to $\mS$ for all $i$.
	\end{itemize} 
\end{thm}

The equivalence between (iii) and (iv) follows from the symmetry 
$\h^i(x_1,\dots,x_s;M)\cong\h_{s-i}(x_1,\dots,x_s;M)$ for 
$0\le i\le s$. The proof of the other equivalences
is provided by a series of lemmas.

\begin{lem}\label{L:I}
	Let $\fa$ be an ideal of $R$ and let $\mC$ be a class of $R$-modules 
	with the follwing three properties:
	\begin{itemize}
		\item[(1)] $\mS\subset\mC$.
		\item[(2)] $0\underset{M}{:}\fa\in\mS$ for each $M\in\mC$.
		\item[(3)] If $0\to{M'}\to M\to M^{\prime\prime}\to 0$ is a 
		short exact sequence with $M^{\prime}$ and $M$ in $\mC$, then also 
		$M^{\prime\prime}$ is in $\mC$.
	\end{itemize}
	Let $X^{\bullet}: 0\to X^0\to X^1\to X^2\to\dots$ be a cochain complex 
	such that for every $i$, $X^i\in\mC$ and 
	$\fa\subset\sqrt{\Ann\h^i(X^{\bullet})}$. 
	Then  $\h^i(X^{\bullet})\in\mS$ for all $i$.
\end{lem}

In order  to prove  Lemma \ref{L:I} we prove two lemmas. 
In those $\mC$ satisfies 
the conditions in Lemma \ref{L:I}.
These lemmas are also of some independent interest.

\begin{lem}\label{L:I_0}
	Let $f: M\to N$ be $R$-linear, where both $M$ and $N$ are in $\mC$. 
	Assume that $\fa^l\Ker f=0$ for some $l$.
	Then $\Ker f$ belongs to $\mS$ {\rm(}hence also to $\mC${\rm)} and 
	$\Coker f$ belongs to $\mC$.
\end{lem}

\begin{proof}
	Since $\Ker f\subset{0\underset{M}{:}\fa^l}$, which belongs to $\mS$ 
	by remark \ref{r:colpow} and the second condition  in 
	Lemma \ref{L:I}. Using the third condition in Lemma \ref{L:I} applied to the 
	exact sequences 
	$0\to\Ker f\to M\to\Image f\to 0$ and 
	$0\to\Image f\to N\to\Coker f\to 0$ 
	we first get that $\Image f\in\mC$ (Note that we assumed that 
	$\mS\subset\mC$) and then that $\Coker f\in\mC$.
\end{proof}

\begin{lem}\label{L:I_n}
	Let $ 0\to X^0\overset{d^0}{\to} X^1\to\dots\to X^{n-1}
	\overset{d^{n-1}}{\to} X^n\overset{d^n}{\to} X^{n+1}$
	be  a cocomplex, where each $X^i$ belongs to $\mC$.
	Suppose that there is some $l$, such that $\fa^l\h^i$ are in $\mS$
	for $i=0,\dots,n$. 
	Then the cohomology modules $\h^i$ are in $\mS$ for $i=0,\dots,n$ and 
	$\Coker d^n\in\mC$.
\end{lem}
\begin{proof}
	We use induction on $n$.
	The case $n=0$ follows from Lemma \ref{L:I_0}.
	Let $n>0$ and assume that we have proved the result for $n-1$.
	The map $d^n$ induces a map
	$f : \Coker{d^{n-1}}\to X^{n+1}$. 
	Here both modules belong to $\mC$.
	The first one by the induction hypothesis and 
	the second one by the hypothesis in the lemma. 
	Since $\Ker f=\h^n$ we can apply Lemma \ref{L:I_0} to prove the case 
	$n$.
\end{proof}

Lemma  \ref{L:I} now follows from  Lemma \ref{L:I_n}.

\begin{lem}\label{L:II}
	Let $\mC$ be a class of $R$-modules satisfying: 
	\begin{itemize}
		\item[(1)] $\mS\subset\mC$.
		\item[(2)] $M/{\fa M}\in\mS$ for all $M\in\mC$.
		\item[(3)] If $0\to M^{\prime}\to M\to M^{\prime\prime}\to 0$ 
		is a short exact with $M$ and $M^{\prime\prime}$ in $\mC$, 
		then also $M^{\prime}\in\mC$.           
	\end{itemize}
	Let 
	$X_{\bullet}:\dots\to X_2\overset{d_1}{\to} X_1\overset{d_0}{\to} X_0\to 0$ 
	be a chain complex 
	such that for all $i$, $X_i\in\mC$ and 
	$\fa\subset\Ann\sqrt{\h_i(X_{\bullet})}$.
	Then the homology modules 
	$\h_i(X_{\bullet})$ belong to $\mS$  for all $i$.
\end{lem}

In the next two lemmas $\mC$ is as in Lemma \ref{L:II}.

\begin{lem}\label{L:II_0}
	Let $f : M\to N$ be an $R$-linear map, where $M$ and $N$ both are in 
	$\mC$. Assume that ${\fa}^l\Coker f=0$ for some $l$. 
	Then $\Ker f$ and $\Coker f$ belong to $\mC$.
\end{lem}

\begin{proof}
	Since $\Coker f$ is a homomorhic image of $N/{\fa}^lN$, it follows from 
	(2) in Lemma \ref{L:II} and remark \ref{r:colpow} that 
	$\Coker f\in\mS$ and hence 
	by (1) in Lemma \ref{L:II} that $\Coker f\in\mC$. 
	By property (3) in Lemma \ref{L:II} using the exact sequences 
	$0\to\Image f\to N\to\Coker f\to 0$  and 
	$0\to\Ker f\to M\to\Image f\to 0$ we first get that 
	$\Image f$ belongs to $\mC$ and then that $\Ker f$ belongs to $\mC$.
\end{proof}

\begin{lem}\label{L:II_n}
	Let $ X_{n+1}\overset{d_n}{\to} X_n\to\dots
	\to X_1\overset{d_0}{\to} X_0\to 0$ be a  complex. 
	Assume that $X_i\in\mC$ for all $i=0,\dots,n$ and
	that there is $l$, such that ${\fa}^l$ annihilates the homology modules
	$\h_i$ for $i=0,\dots,n$. 
	Then $\Ker d_n$ belongs to  $\mS$ {\rm(}hence also to $\mC${\rm)}  
	and the homology modules 
	$\h_i$ belong to $\mS$ for 
	$i=0,\dots,n$.
\end{lem}
\begin{proof}
	The proof is by induction on $n$.
	The case case $n=0$ is just Lemma \ref{L:II_0}.
	Assume that $n>0$ and that we know that 
	$\Ker d_{n-1}$ belongs to $\mC$ and that the homology modules 
	$\h_i$ belong to
	$\mS$ for $i=0,\dots,n-1$.
	Consider the map $f : X_{n+1}\to \Ker d_{n-1}$ induced by $d_n$. 
	Since $\Ker f=\Ker d_n$ and $\Coker f=\h_{n-1}$ we get by \ref{L:II_0} that 
	$\Ker d_n\in\mC$ and that $\h_{n-1}\in\mS$.   
\end{proof}
We now prove the equivalences in Theorem \ref{T:E}. 
\begin{proof}[Proof of Theorem~\ref{T:E}]
	We have already noticed that the conditions (iii) and (iv) are equivalent.
	Let ${\mC}_i$, ${\mC}_{ii}$, ${\mC}_{iii}$ and ${\mC_{iv}}$ be the classes of 
	modules which satisfy the respective condition in Theorem \ref{T:E}.
	We have already noticed that $\mC_{iii}$ and $\mC_{iv}$ are equal.
	Let $F_{\bullet}\to R/\fa$ be a resolution of $R/\fa$ consisting of 
	finite free modules and let $\K_{\bullet}=\K_\bullet (x_1,\dots,x_s)$ be the 
	Koszul complex  of the sequence $x_1,\dots,x_s$. 
	The classes $\mC_i$ and $\mC_{iii}$ 
	both  satisfy the three  conditions in \ref{L:I}.
	If $M$ is in $\mC_i$ then  by \ref{L:I} the cohomology modules of the
	cochain complex $X^{\bullet}=\Hom_R(\K_{\bullet}, M)$ 
	are in $\mS$ and therefore $M$ is in $\mC_{iii}$.
	Conversely let $M$ be in $\mC_{iii}$. 
	Consider the cochain complex 
	$X^{\bullet}=\Hom_R(F_{\bullet},M)$. 
	We get using Lemma \ref{L:I}, that $M$   belongs to $\mC_i$. 
	To prove that the classes $\mC_{ii}$ and $\mC_{iv}$ are the same  
	we consider the complexes $\K_{\bullet}\otimes M$ and 
	$F_{\bullet}\otimes M$ and use Lemma~\ref{L:II}. 
\end{proof}

\begin{thm}\label{T:ET}
	For each $R$-module $M$ 
	which satisfies the equivalent conditions of Theorem \ref{T:E}
	and each finite $R$-module $N$ with support 
	in $\V(\fa)$, the modules 
	$\Ext_R^i(N, M)$ and $\Tor_i^R(N,M)$ belong to $\mS$.  
\end{thm}
\begin{proof}
	Let $\mC$ be the class of $R$-modules which satidfy the
	equivalent conditions in Theorem \ref{T:E}. 
	Apply Lemma \ref{L:I} and Lemma \ref{L:II}
	to the complexes $X^{\bullet}=\Hom_R(F_{\bullet},M)$ and 
	$X_{\bullet}=F_{\bullet}\otimes M$.
\end{proof}


\section{Cofiniteness of local cohomology and generalized local cohomology}


The following two result generalize and decrease the assumptions on $M$ in \cite[Theorem 3.4 (i) and Corollary 3.5 (i)]{AB}. 

\begin{thm}
	\label{homext} Let $R$ be a Noetherian ring and $\fa$ an ideal of $R$. Let $t\in\Bbb{N}_0$ be an integer and $M$ an $R$-module such that $\Ext^i_R(R/\fa,M)$ are finite for all $i\leq t+1$. Let the $R$-modules $\lc^i_\fa(M)$ are ${\rm FD_{<2}}$ ~ $R$-modules for all $i<t$.
	Then, the $R$-modules $\lc^i_\fa(M)$ are $\fa$-cofinite for all
	$i<t$.
\end{thm}
\proof  We use induction on $t$. The exact sequence 

$$0\lo \G_\fa(M)\lo M \lo M/\G_\fa(M)
\lo 0,\,\,\,\,\,\,(*)$$

induces the following exact sequence:

$$0\longrightarrow {\Hom}_R(R/\fa,\G_\fa(M))\lo{\Hom}_R(R/\fa,M)\lo {\Hom}_R(R/\fa,M/\G_\fa(M))$$$$\lo {\Ext}^1_R(R/\fa,\G_\fa(M))\lo{\Ext}^1_R(R/\fa,M).$$

Since ${\Hom}_R(R/\fa,M/\G_\fa(M))=0$ so ${\rm
	Hom}_R(R/\fa,\G_\fa(M))$ and ${\rm Ext}^{1}_R(R/\fa,\G_\fa(M))$ are finite. Assume inductively that $t>0$ and that we have established the result for non-negative integers smaller than $t$. Using inductive assumption, the rest of proof is the same as \cite[Theorem 3.4 (i)]{AB}.\qed\\


\begin{cor}
	\label{cof} Let $R$ be a Noetherian ring and $\fa$ an ideal of $R$. Let $M$ be
	an $R$-module such that the $R$-module  $\Ext^i_R(R/\fa ,M)$ is finite for all $i$ and the $R$-module $\lc^i_\fa(M)$ is ${\rm
		FD_{<2}}$ {\rm{(}}or weakly Laskerian{\rm{)}}~ $R$-modules for all $i$. Then the $R$-modules $\lc^i_\fa(M)$ are $\fa$-cofinite for all $i$.

\end{cor}
\proof  Follows by Theorem \ref{homext}.\qed\\

 The following corollary is a characterization of cofiniteness of local cohomology modules which are ${\rm
 	FD_{<2}}$ $R$-module. 
\begin{cor}
	Let $R$ be a Noetherian ring and $\fa$ an ideal of $R$. Let $M$ be
	an $R$-module such that  the $R$-modules $\lc^i_\fa(M)$ are ${\rm
		FD_{<2}}$ {\rm{(}}or weakly Laskerian{\rm{)}}~ $R$-modules for all $i$. Then, the following conditions are equivalent:
	
	{\rm(i)} The $R$-module  $\Ext^i_R(R/\fa ,M)$ is finite for all $i$.
	
	{\rm(ii)} The $R$-module $\lc^i_\fa(M)$ is $\fa$-cofinite for all $i$.	
\end{cor}

\proof (i)$\Rightarrow$(ii) Follows by  Corollary \ref{cof}.

(ii)$\Rightarrow$(i) It follows by \cite[Proposition 3.9]{Mel}.\qed\\


\begin{thm}\label{cof3}
	Let $M$ be an $R$-module and suppose one of the following cases holds:
	\begin{itemize}
		\item[(a)] $\ara (\fa)\leq 1$;
		\item[(b)]  $\dim R/\fa \leq 1$;
		\item[(c)]  $\dim R\leq 2$.
	\end{itemize}
	Then,  $\lc^{i}_\fa(M)$ is $\fa$-cofinite for all $i$ if and only if $\Ext^i_R(R/\fa ,M)$ is a finite $R$-module for all $i$.
\end{thm}

\begin{proof}
	The	case (c) follows by Theorem \cite[Theorem 7.10]{Mel}. 
	
	In the cases (a) and (b), suppose  $\lc^{i}_\fa(M)$ is $\fa$-cofinite for all $i$. It follows from \cite[Corollary 3.10]{Mel} that $\Ext^i_R(R/\fa ,M)$ is a finite $R$-module for all $i$. 
	
	To prove the converse, in the case (b), it follows by \cite[Theorem 2.10]{Mel1}. In the case (a), use \cite[Theorem 3.3]{BA} and the method of proof of \cite[Theorem 2.10]{Mel1}.
\end{proof}

\begin{lem} \label{1}
	Let $X$ be a finite $R$--module, $M$ be an arbitrary $R$--module. Then the following statements hold true.
	\begin{itemize}
		\item[(a)] $\Gamma_{\fa}(X, M)\cong \Hom_{R}(X, \Gamma_\fa (M)).$
		\item[(b)] If $\Supp_R(X)\cap \Supp_R(M)\subseteq \V(\fa)$, then $\lc^{i}_{\fa}(X, M)\cong\Ext^i_{R}(X, M)$ for all $i$.
	\end{itemize}
\end{lem}

\begin{proof}
	See \cite[Lemma 2.5]{VA}.
\end{proof}


\begin{thm}\label{cof4}
	Let $M$ be an $R$-module and suppose one of the following cases holds:
	\begin{itemize}
		\item[(a)] $\ara (\fa)\leq 1$;
		\item[(b)]  $\dim R/\fa \leq 1$;
		\item[(c)]  $\dim R\leq 2$.
	\end{itemize}
	Then, for any finite $R$-module $N$, $\lc^{i}_\fa(N,M)$ is $\fa$-cofinite for all $i\geq 0$ if and only if $\Ext^i_R(R/\fa ,M)$ is finite $R$-module  for all $i\geq 0$.
\end{thm}

\begin{proof}
	First suppose for any finite $R$-module $N$, $\lc^{i}_\fa(N,M)$ is $\fa$-cofinite for all $i\geq 0$. Let $N=R$, then it follows by Theorem \ref{cof3}, that $\Ext^i_R(R/\fa ,M)$ is finite $R$-module  for all $i\geq 0$.
	
	To prove the converse we use induction on $i$. Let $i=0$, then it follows by Lemma \ref{1} (a) that
	\begin{center}
		$\lc^{0}_\fa(N,M)=\Gamma_{\fa}(N, M)\cong \Hom_{R}(N, \Gamma_\fa (M)).$
	\end{center}	
	
	Since $\Gamma_{\fa}(M)$ is $\fa$-cofinite by Theorem \ref{cof3},
	so the assertion follows by \cite[Theorem 4.1]{BA}, \cite[Corollary 2.7]{Mel1} and \cite[Theorem 7.4]{Mel}, respectively in the cases (a), (b) and (c). Now assume that $i> 0$ and that the claim holds for $i-1$. Since $\G_{\fa}(M)$ is $\fa$-cofinite and $\Ext^i_R(R/\fa ,M)$ is finite $R$-module for all $i\ge 0$. Using the short exact sequence
	
	\begin{center}
		$0\lo \G_{\fa}(M)\lo M \lo M/\G_{\fa}(M)
		\lo 0,$
	\end{center}
	
	it is easy to see that  the $R$-modules ${\Ext}^{i}_R(R/{\fa},M/\G_{\fa}(M))$ are finite for all $i\ge 0$. Now, by applying the derived functor $\Gamma_{\fa}(N,-)$ to the same short exact sequence and using Lemma \ref{1} (b), we obtain the long exact sequence 
	
	\begin{center}
		$\dots \lo \Ext_R^i(N,\G_{\fa}(M))\overset{f_i} 
		\lo \lc^{i}_\fa(N,M)\overset{g_i}\lo \lc^{i}_\fa(N,M/\G_{\fa}(M))\overset{h_i}\lo \Ext_R^i(N,\G_{\fa}(M))\overset{f_{i+1}}\lo\lc^{i+1}_\fa(N,M)\lo\dots$
	\end{center}
	
	which yields short exact sequences 
	
	\begin{center}
		$0\lo \Ker{f_i}\lo \Ext_R^i(N,\G_{\fa}(M)) \lo \Image{f_i}
		\lo 0,$
	\end{center}
	
	\begin{center}
		$0\lo \Image{f_i}\lo \lc^{i}_\fa(N,M) \lo \Image{g_i}
		\lo 0$
	\end{center}
	
	\begin{center}
		$\text{and}$
	\end{center}
	
	\begin{center}
		$0\lo \Image{g_i}\lo \lc^{i}_\fa(N,M/\G_{\fa}(M)) \lo \Ker{f_{i+1}}
		\lo 0.$
	\end{center}
	
	Since $\Ext_R^i(N,\G_{\fa}(M))$	is  $\fa$-cofinite $R$-module for all $i\ge 0$ by \cite[Theorem 4.1]{BA}, \cite[Corollary 2.7]{Mel1} and \cite[Theorem 7.4]{Mel}, respectively in the cases (a), (b) and (c), it follows by definition that, $\lc^{i}_\fa(N,M)$ is $\fa$-cofinite $R$-module if and only if $\lc^{i}_\fa(N,M/\G_{\fa}(M))$ is  $\fa$-cofinite $R$-module for all $i\ge 0$. Therefore it suffices to show that $\lc^{i}_\fa(N,M/\G_{\fa}(M))$ is  $\fa$-cofinite $R$-module for all $i\ge 0$. To this end consider the exact sequence 
	
	\begin{center}
		$0\lo M/\G_{\fa}(M)\lo E \lo L
		\lo 0,\,\,\,\,\,\,\,\,(\dagger)$
	\end{center}
	
	in which $E$ is an injective $\fa$-torsion free module. Since 
	$\G_{\fa}(M/\G_{\fa}(M))= 0= \G_{\fa}(E)$, thus $\Hom_R(R/\fa,E)=0$ and $\G_{\fa}(N, M/\G_{\fa}(M))= 0= \G_{\fa}(N, E)$ by Lemma \ref{1} (a). Applying the derived functors of $\Hom_R(R/\fa,-)$ and $\Gamma_{\frak{a}}(N,-)$ to the short exact sequence $(\dagger)$
	we obtain, for all $i> 0$, the isomorphisms
	
	\begin{center}
		$\lc^{i-1}_\fa(N, L)\cong \lc^i_\fa(N,M/\G_{\fa}(M))$.
		
	\end{center}
	
	\begin{center}
		\text{and} 	
	\end{center}	 
	
	\begin{center}
		${\Ext}^{i-1}_R(R/{\fa},L)\cong {\Ext}^{i}_R(R/{\fa},M/\G_{\fa}(M))$.
	\end{center}
	
	From what has already been proved, we conclude that  ${\Ext}^{i-1}_R(R/{\fa},L)$ is finite for all $i>0$. Hence $\lc^{i-1}_\fa(N, L)$ is $\fa$-cofinite by induction hypothesis for all $i>0$, which yields that $\lc^i_\fa(N,M/\G_{\fa}(M))$ is $\fa$-cofinite for all $i\ge 0$, this completes the inductive step. 
\end{proof}


\begin{lem}\label{mar}
	
	Let $N$ be a finitely generated $R$-module and $M$ be an arbitrary $R$-module. Let $t$ be a non-negative integer such that ${\Ext}_R^i(N,M)$ is finite for all $0\leq i\leq t$. Then for 
	any finitely generated $R$-module $L$ with
	$\Supp_R(L)\subseteq \Supp_R(N)$, ${\Ext}_R^i(L,M)$ is finite for all integer $0\leq i\leq t$.
\end{lem}

\begin{proof}
See \cite[Proposition 1]{DM}.
\end{proof}	


The following corollary is our main result of this section which is a characterization of  cofinite local cohomology and generalized local cohomology modules under the assumptions (a) $\ara (\fa)\leq 1$, (b) $\dim R/\fa \leq 1$ and (c) $\dim R\leq 2$. This corollary shows that with the above assumptions, cofiniteness of local cohomology and generalized local cohomology modules are equivalent.

\begin{cor}\label{min}
	Let $M$ be an $R$-module and suppose one of the following cases holds:
	\begin{itemize}
		\item[(a)] $\ara (\fa)\leq 1$;
		\item[(b)]  $\dim R/\fa \leq 1$;
		\item[(c)]  $\dim R\leq 2$.
	\end{itemize}
	Then the following conditions
	are equivalent: 
	\begin{itemize}
		\item[(i)] $\Ext^i_R(R/\fa ,M)$ is finite for all $i$.
		\item[(ii)] $\Ext^i_R(R/\fa ,M)$ is finite for $i=0,1$ in the cases {\rm(}a{\rm)} and {\rm(}b{\rm)} {\rm(}resp. for $i=0,1,2$  in the case {\rm(}c{\rm)}{\rm)};
		\item[(iii)] $\lc^{i}_\fa(M)$ is  $\fa$-cofinite  for all $i$;
		\item[(iv)]  $\lc^{i}_\fa(N,M)$ is  $\fa$-cofinite
		for all $i$ and for any finite $R$-module $N$;
		\item[(v)] $\Ext^{i}_R(N,M)$ is  finite  for all $i$ and for any finite $R$-module $N$ with $\Supp_R(N)\subseteq \V(\fa)$;
		\item[(vi)] $\Ext^{i}_R(N,M)$ is  finite  for all $i$ and for some finite $R$-module $N$ with $\Supp_R(N)=\V(\fa)$;
		\item[(vii)] $\Ext^{i}_R(N,M)$ is  finite  for $i=0,1$ in the cases {\rm(}a{\rm)} and {\rm(}b{\rm)} {\rm(}resp. for $i=0,1,2$  in the case {\rm(}c{\rm)}{\rm)} and for any finite $R$-module $N$ with $\Supp_R(N)\subseteq \V(\fa)$;
		\item[(viii)] $\Ext^{i}_R(N,M)$ is  finite  for $i=0,1$ in the cases {\rm(}a{\rm)} and {\rm(}b{\rm)} {\rm(}resp. for $i=0,1,2$  in the case {\rm(}c{\rm)}{\rm)} and for some finite $R$-module $N$ with $\Supp_R(N)=\V(\fa)$.
	\end{itemize}
\end{cor}

\begin{proof}
	In order to prove (i)$\Leftrightarrow$(ii),  use \cite[Theorem 3.3]{BA}, \cite[Theorem 2.3]{Mel1} and \cite[Theorem 7.10]{Mel}, respectively in the cases (a), (b) and (c).
	
	(i)$\Leftrightarrow$(iii) follows by Theorem \ref{cof3}.
	
	(i)$\Leftrightarrow$(iv) follows by Theorem \ref{cof4}.
	
	In order to prove (i)$\Leftrightarrow$(v) and (ii)$\Leftrightarrow$(vii) use  \cite[Lemma 1]{Ka}.
	
	(v)$\Rightarrow$(vi) and (vii)$\Rightarrow$(viii) are trivial.
	
	In order to prove (vi)$\Rightarrow$(v) and (viii)$\Rightarrow$(vii), let $L$ be a finitely generated $R$-module with $\Supp_R(L)\subseteq \V(\fa)$. Then $\Supp_R(L)\subseteq \Supp_R(N)$. Now the assertion follows by Lemma \ref{mar}.
\end{proof}

It is easy to see that $\cd(\fa,R)\leq \ara \fa$ but the equality is not true in general (see \cite[Example 2.3]{DFT}). We close this section by offering a question and problem for further research.\\

\noindent {\bf Question:} Let $R$ be a commutative Noetherian ring with non-zero identity and
$\fa$ an ideal of $R$. Is the charaterization  in Corollary \ref{min} true, when we change $\ara \fa\leq 1$ with $\cd(\fa,R)\leq 1$ in the case (a)?





\section{Artinianness of local cohomology modules} 


	

	


In general it is not true that when $M$ is a finite $R$-module over the Noetherian ring $R$, such that the local cohomology modules $\lc^{i}_\fa(M)$ have support in $\Max R$ for all $i>r$, then they are Artinian $R$-module for all $i>r$. For this see the example by Hartshorne in \cite[page 151]{Har}. However we prove that this is true in the case when $\dim R/{\fa}=1$. To do this we need the following lemma.


\begin{lem}\label{sez3.11}
	Let $\fa$ be an ideals of a Noetherian ring $R$, such that $\dim R/{\fa}=1$. Let $M$ be a finite $R$-module and $r$ an integer. Then $\dim M_{\fp}\leq r$  for each prime ideal  $\fp$ minimal over $\fa$  if and only if $\Supp_R(\lc^{i}_\fa(M))\subseteq \Max(R)$ for all $i>r$.
\end{lem}

\begin{proof}
This is the direct consequence of Grothendieck's vanishing theorem and the isomorphism
$\lc^{i}_\fa(M)_{\fp}\cong \lc^{i}_{{\fa}R_{\fp}}(M_{\fp})$ for each prime ideal $\fp$ minimal over $\fa$.	
\end{proof}


\begin{thm}\label{sez3.12}
	Let $\fa$ be an ideals of a Noetherian ring $R$, with $\dim R/{\fa}=1$. For an $R$-module $M$   and  integer  an  $r$ the following conditions are equivalent; 

	\begin{itemize}
		\item[(i)]$\Supp_R(\lc^{i}_\fa(M))\subseteq \Max(R)$ for each $i>r$.
		\item[(ii)] $\lc^{i}_\fa(M)$ is an Artinian $R$-module for each $i>r$.
		\item[(iii)] $\lc^{i}_\fa(M)$ is an Artinian $\fa$-cofinte $R$-module for each $i>r$.
	\end{itemize} 
\end{thm}

\begin{proof} The  implications (iii)$\Rightarrow$(ii)$\Rightarrow$(i) always hold, so let us assume (i) and prove that (iii) hold. We do induction on $r$. First we note that if $N$ is a finite $R$-module, such that $\Supp_R(\lc^{i}_\fa(N))\subseteq \Max(R)$, then $\lc^{i}_\fa(N)$ is $\fa$-cofinite Artinian for all $i$. Let $r=0$ and put $N=M/\G_{\fa}(M)$. Then $\lc^{0}_\fa(N)=0$ and  $\lc^{i}_\fa(N)\cong \lc^{i}_\fa(M)$ for all $i>0$.  Hence by the above remark $\lc^{i}_\fa(N)$ is $\fa$-cofinite Artinian for all $i$. It follows that $\lc^{i}_\fa(M)$ is Artinian and $\fa$-cofinite for all $i>0$. 
	
Next let $r>0$	and put $N=M/\G_{\fa}(M)$. Then $\lc^{0}_\fa(N)=0$ and we can take $x\in \fa$ such that $x$ is $N$-regular.

Then by Lemma \ref{sez3.11}, $\dim(N/xN)_{\fp}\leq r-1$ for each prime ideal $\fp$ minimal over $\fa$. Applying the induction hypothesis to $N/xN$, we have that $\lc^{i}_\fa(N/xN)$ is $\fa$-cofinite Artinian for all $i>r-1$. Next applying local cohomology functor to the exact sequence $0\lo N\lo N \lo N/xN
	\lo 0,$ we can conclude that $(0:_{\lc^{i}_\fa(N)}x)$ is $\fa$-cofinite Artinian for all $i>r$. It follows that $\lc^{i}_\fa(N)$ is Artinian and $\fa$-cofinite for all $i>r$. Therefore also $\lc^{i}_\fa(M)$ , for all $i>r$.

	
	
	
\end{proof}





\bibliographystyle{amsplain}

\end{document}